\documentclass[11pt]{article}
\usepackage{amsmath,amsfonts,amsthm}
\usepackage{amssymb}
\usepackage{fullpage}
\usepackage{enumerate}
\usepackage{graphicx}

\newtheorem{theorem}{Theorem}
\newtheorem{lemma}[theorem]{Lemma}
\newtheorem{corollary}[theorem]{Corollary}

\newcommand\RR{\ensuremath{\mathbb{R}}}
\newcommand\NN{\ensuremath{\mathbb{N}}}

\newcommand\Val{\mathop{{\hbox{\sc Val}}}}
\newcommand\ValC{\underline{\Val}_{\,C}}
\newcommand\ValR{\overline{\Val}_{R}}

\begin{document}

\title{Min-max theorem for the game of Cops and Robber on geodesic spaces}

\author{Bojan Mohar\thanks{Supported in part by the NSERC Discovery Grant R611450 (Canada),
and by the Research Project J1-2452 of ARRS (Slovenia).}%
~\thanks{On leave from IMFM, Department of Mathematics, University of Ljubljana.}\\
\small Department of Mathematics\\[-0.8ex]
\small Simon Fraser University\\[-0.8ex]
\small Burnaby, BC \ V5A 1S6, Canada\\
\small\tt mohar@sfu.ca}

\date{}

\maketitle

\begin{abstract}
The game of Cops and Robber is traditionally played on a finite graph. The purpose of this note is to introduce and analyze the game that is played on an arbitrary geodesic space. The game is defined in such a way that it preserves the beauty and power of discrete games played on graphs and also keeps the specialties of the pursuit-evasion games played on polyhedral complexes. It is shown that the game can be approximated by finite games of discrete type and as a consequence a min-max theorem is obtained.
\end{abstract}


\section{Introduction}

Pursuit-evasion games have a long history, especially in the setup of \emph{differential games} \cite{Is65,Ji15,Le94,Pa70,Pe93}. Differential games with more pursuers were also introduced in the 1970s, see, e.g. \cite{HaBr74,Ch76,Ps76,LePa85} or a more recent paper \cite{FeIbAlSa20} and references therein. More recent important application is design of robot movement in complicated environment, see e.g. \cite{AHRW17}. A more general class of such games played on finite graphs has been devised in discrete setting. Nowakowski and Winkler \cite{NoWi83} and Quilliot \cite{Qui78} independently introduced the game of Cop and Robber that is played on a (finite) graph. Aigner and Fromme \cite{AiFr84} extended the game to include more than one cop. For each graph $G$ and a positive integer $k$, the \emph{Cops and Robber game} on $G$, involves two players. The first player controls $k$ \emph{cops} placed at the vertices of the graph, and the second player controls the \emph{robber}, who is also positioned at some vertex. While the players alternately move to adjacent vertices (or stay at their current position), the cops want to catch the robber and the robber wants to prevent this ever to happen. The main question is how many cops are needed on the given graph $G$ in order that they can guarantee the capture. The minimum such number of cops is termed as the \emph{cop number} $c(G)$ of the graph.

The game of cops and robbers gained attention because of its ties with structural graph theory. Classes of graphs that can be embedded in a surface of bounded genus \cite{AiFr84} and those that exclude some fixed graph as a minor \cite{An86} have bounded cop number. In particular, all graphs that can be embedded in the plane have cop number at most 3 \cite{AiFr84}. We refer to the monograph by Bonato and Nowakowski \cite{BoNo11} for further details about the history of the game and for overview of the main results.

One of our aims is to introduce the game in a more general setup of geodesic metric spaces and study the relationship between the cop number and the topology and geometry of the geodesic space.

The famous Lion and Man problem that was proposed by Richard Rado in the late 1930s and discussed in Littlewood's Miscellany \cite{Li53,Li86} is a version of the game with one pursuer (the Lion) and one evader (the Man). The man and the lion are within a circular arena (unit disk in the plane) and they run with equal maximum speed. It seems that in order to avoid the lion, the man would choose to run on the boundary of the disk. A simple argument then shows that the lion could always catch the man by staying on the segment joining the center of the disk with the point of the man and slowly approaching him. However, Besicovitch proved in 1952 (see \cite[pp.~114--117]{Li86}) that the man has a simple strategy, in which he will approach but never reach the boundary, that enables him to avoid capture forever no matter what the lion does.\footnote{The game defined in this paper allows the use of Besicovitch strategy for the man, so this example shows that the lion is able to come arbitrarily close to the man, but can never catch him.} More details can be found in~\cite{BoLeWa12}.

One can prove that two lions are enough to catch the man in a disk. A recent work by Abrahamsen et al.\ \cite{AHRW17,AHRW20} discusses the game with many lions versus one man in an arbitrary compact subset of the plane whose boundary consists of finitely many rectifiable simple closed curves and prove that three lions can always get their prey. They also discuss the game when the man is just slightly faster than lions, and find some surprising conclusions.

The game of cops and robbers can be defined on any metric space.
However, it is far from obvious how such a game can be defined in order to be natural, resembling interesting examples and allowing for powerful mathematical tools. Subtleties of the various versions of the game are nicely outlined in an influential paper by Bollob\'as, Leader, and Walters \cite{BoLeWa12}, who were the first to provide a general setup for such a game.

In this article we discuss the game of cops and robbers on arbitrary geodesic spaces (see Section \ref{sect:geodesic spaces} or \cite{BuBuIv01,BuSh04} for definitions). We come up with a version of the game that is somewhat different from the game version in \cite{BoLeWa12}, but preserves all the beauty and power of discrete games played on graphs.
Moreover, our version keeps the characteristics of the pursuit-evasion games played in a continuous setting and for instance allows for using strategies similar to that of Besicovitch in the case of the Man and Lion game. It is shown that our game can be approximated by finite games of discrete type and as a consequence we are able to prove the min-max theorem.

\section{Intrinsic metric in geodesic spaces}
\label{sect:geodesic spaces}

We consider a metric space $(X,d)$ and the corresponding metric space topology on $X$. For $x,y\in X$, an \emph{$(x,y)$-path} is a continuous map $\gamma: I\to X$ where $I=[0,1]$ is the unit interval on $\RR$ and $\gamma(0)=x$ and $\gamma(1)=y$.
We allow the paths to be parametrized differently and in particular we can replace $I$ with any finite interval on $\RR$.
The space is \emph{path-connected} if for any $x,y\in X$, there exists an $(x,y)$-path connecting them.

One can define the \emph{length} $\ell(\gamma)$ of the path $\gamma$ by taking the supremum over all finite sequences $0=t_0<t_1<t_2< \cdots < t_n=1$ of the values $\sum_{i=1}^n d(\gamma(t_{i-1}),\gamma(t_i))$. Note that $\ell(\gamma)$ may be infinite; if it is finite, we say that $\gamma$ is \emph{rectifiable}. Note that the length of any $(x,y)$-path is at least $d(x,y)$.  The metric space $X$ is a \emph{geodesic space} if for every $x,y\in X$ there is an $(x,y)$-path $\gamma$ whose length is equal to $d(x,y)$.

An $(x,y)$-path $\gamma$ is \emph{isometric} if $\ell(\gamma) = d(x,y)$. Observe that for $0\le t < t' \le 1$ the subpath $\gamma|_{[t,t']}$ is also isometric. Therefore the set $\gamma(I) = \{\gamma(t)\mid t\in I\}$ is an isometric subset of $X$. With a slight abuse of terminology, we say that the image $\gamma(I)\subset X$ is an \emph{isometric path} in $X$.

A path $\gamma$ is a \emph{geodesic} if it is locally isometric, i.e., for every $t\in [0,1]$ there is an $\varepsilon>0$ such that the subpath $\gamma|_J$ on the interval $J = [t-\varepsilon,t+\varepsilon]\cap[0,1]$ is isometric. A path with $\gamma(0)=\gamma(1)$ is called a \emph{loop} (or a \emph{closed path}). When we say that a loop is a geodesic, we mean it is geodesic as a path and it is also locally isometric around its base point, i.e. $\gamma|_{[1-\varepsilon,1]\cup[0,\varepsilon]}$ is isometric for some $\varepsilon>0$.

Alternatively, one can consider any path-connected compact metric space $X$ and then define the shortest path distance. For $x,y\in X$, the \emph{shortest path distance} from $x$ to $y$ is defined as the infimum of the lengths of all $(x,y)$-paths in $X$. If any two points in $X$ are joined by a path of finite length, then the shortest path distance gives the same topology on $X$. By the Arzel\`a-Ascola theorem (see e.g. \cite{BuBuIv01}), compactness implies that any sequence of $(x,y)$-paths contains a point-wise convergent subsequence, and that the limit points determine an $(x,y)$-path. This implies that there is a path whose length is equal to the infimum of all path lengths. Hence, for this metric, which is also known as the \emph{intrinsic metric}, $X$ is a geodesic space.

If $X$ is a geodesic space, each of its points appears on a geodesic. But some points only appear as the end-points of isometric paths in $X$ and cannot appear as interior points of those. Such points are called \emph{corners}. All other points appear as internal points on geodesics in $X$ and are said to be \emph{regular points} in $X$. It is obvious that regular points are dense in $X$. On the other hand, the set of corners can also be very rich. It may contain the whole boundary component, but in the interior of $X$, it is totally path-disconnected in the sense that every path containing only corners is either trivial (a single point), or is contained in $\partial X$.

Common examples of geodesic spaces include any connected cell complex endowed with the intrinsic metric.
If a geodesic space is homeomorphic to a 1-dimensional cell complex (graph), then we say that it is a \emph{metric graph}.
If $G$ is a graph and $w:E(G)\to \RR_+$ is a function specifying the length of each edge, we define the \emph{metric graph} $X(G,w)$ corresponding to $G$ and $w$ as the metric graph $G$ in which each edge $e$ is represented by a real interval of length $w(e)$.

We refer to \cite{BuBuIv01} and \cite{BuSh04} for a thorough treatment of geodesic spaces.

\section{Game of Cops and Robber on geodesic spaces}

\subsubsection*{Rules of the game}

Let $X$ be a compact geodesic space endowed with intrinsic metric $d$, and let $k\ge1$ be an integer. A \emph{Game of Cops and Robber} on the \emph{game space} $X$ with $k$ cops is a two-person game with complete information defined as follows. The first player controls an avatar, who is positioned at a point $r\in X$ and whom we call the \emph{robber}. The second player controls a set of $k$ \emph{cops} $C_1,\dots,C_k$ that are also positioned in $X$. It is allowed that different cops occupy the same position in $X$. There are rules how the game starts and how the players move, and the goal of the second player is to come as close as possible to the robber (possibly catching him, i.e. one of the cops to occupy the same point in $X$ as the robber). The details about these rules are specified below.

A \emph{position} in a game with $k$ cops is a $(k+1)$-tuple $(r,c_1,\dots,c_k)\in X^{k+1}$ enlisting the positions of the robber and the cops. Instead of $(r,c_1,\dots,c_k)$ we also write $(r,c)$, where $c=(c_1,\dots,c_k)\in X^k$. The game is defined by the following parameters (in addition to the game space $X$ and $k$):

\begin{description}
  \item[(I)] A rule that specifies the set of \emph{admissible initial positions} of the robber and the cops. This is just a set of $(k+1)$-tuples, $\mathcal Y^0 \subseteq X^{k+1}$.
  \item[(S)] A set $\Sigma^0$ of \emph{agility functions}, each of which maps $\NN\to\RR_+$.
\end{description}
The \emph{standard game} $\Gamma_0$ has $\mathcal Y^0 = X^{k+1}$ and $\Sigma^0$ contains all positive functions $\tau:\NN\to\RR_+$ for which $\sum_{n\ge1} \tau(n) = \infty$. Throughout this paper we will stick with these assumptions unless stated differently.

Given (I) and (S), the robber selects an initial position $Y^0 = (r^0,c_1^0,\dots,c_k^0) \in \mathcal Y^0$ and selects his agility $\tau\in \Sigma^0$.
Then the game proceeds as a discrete game in consecutive steps. Having made $n-1$ steps $(n\ge1)$, the players are in position $(r^{n-1},c_1^{n-1},\dots,c_k^{n-1})\in X^{k+1}$. The $n$th step will have its duration determined by the agility: the move will last for time $\tau(n)$, and each player can move with unit speed up to a distance at most $\tau(n)$ from his current position.
First, the robber moves to a point $r^n\in X$ at distance at most $\tau(n)$ from its current position, i.e. $d(r^{n-1}, r^n)\le \tau(n)$. The destination $r^n$ is revealed to the cops. Then each cop $C_i$ ($i\in[k]$) selects his new position $c_i^n$ at distance at most $\tau(n)$ from its current position, i.e. $d(c_i^{n-1}, c_i^n)\le \tau(n)$. The game stops if $c_i^n = r^n$ for some $i\in[k]$. In that case, the \emph{value of the game} is 0 and we say that the cops \emph{have caught} the robber. Otherwise the game proceeds with the next step. If it never stops, the \emph{value of the game} is
\begin{equation}\label{eq:value of game}
v = \inf_{n\ge0} \min_{i\in[k]} d(r^n, c_i^n).
\end{equation}
If the value is 0, we say that the \emph{cops won} the game; otherwise the \emph{robber wins}. Note that the cops can win even if they never catch the robber.\footnote{Consider the afore-mentioned strategy of Besicovitch \cite{BoLeWa12} for the game of Lion and Man.}

The traditional description of pursuit-evasion games starts by the cops first choosing their position and then the robber choosing his. This setting is actually equivalent to our standard game, since the cops can always move to their desired initial positions during the beginning of the game.

We will also consider other variants of the game:
\begin{itemize}
  \item
  The \emph{game $\Gamma(r,c)$ with the fixed initial position $(r,c)$}.  Here we have $\mathcal Y^0 = \{(r,c)\}$.
  \item
  The \emph{game $\Gamma(\tau)$ with the fixed agility} $\tau$, and its version where also the initial position is fixed, $\Gamma(r,c,\tau)$. Here we have $\Sigma^0 = \{\tau\}$ (and $\mathcal Y^0 = \{(r,c)\}$), and we allow that $T = \sum_{n\ge1} \tau(n)$ is either finite or infinite.
  \item A \emph{finite $N$-step game} $\Gamma(N,\tau)$ and its version with the initial position fixed, $\Gamma(r,c,N,\tau)$. The game stops after $N$ steps. Here, only the first $N$ values of $\tau$ are important for the game, so we may assume that $\tau: [N] \to \RR^+$ or that $\tau(n)=0$ for $n > N$.
  \item A \emph{finite-time game} $\Gamma(T)$ or $\Gamma(r,c,T)$ is a finite-step game, where the constraint is not the number of steps but the total duration, i.e., the agility functions satisfy $\sum_{n\ge1} \tau(n) = T$. Here $T$ is a positive real number, the \emph{duration of the game}. This version may be combined with fixing the number of steps, $\Gamma(N,T)$ or even fixing the agility, $\Gamma(N,\tau,T)$, where we ask $\sum_{n=1}^N \tau(n) = T$; additionally, we can fix the initial position.
\end{itemize}

It will always be clear by the listed set of parameters which version of the game we have in mind.

In general, we should also add the game space $X$ and the number $k$ of cops among the parameters, but since this is almost always implicit from the context, we usually omit them.

\subsubsection*{Strategies and value of the game}

The value of the game when played in $X$ is defined by (\ref{eq:value of game}). Note that the strategy of a player depends not only on the current position but also on the agility chosen by the robber at the very beginning. To formalize this dependence, we introduce the notion of a strategy.

The strategy of the robber is first to select an initial position and agility. This is a formal part of his strategy. The rest of his strategy and a strategy of the cops may be defined via a game with a fixed initial position and fixed agility, $\Gamma(r,c,\tau)$. Formally (and leaving out the initial choice of the robber), a \emph{strategy of the robber} is a function $s: (r,c,\tau) \mapsto r'$, such that $d(r,r')\le \tau(1)$. This can be interpreted as moving the robber from $r$ to $r'$ along some geodesic of length $d(r,r')$. Then, each cop $C_i$ moves from his current position $c_i$ to a point $c_i'$ at distance at most $\tau(1)$ from $c_i$. The choice of such destinations $c' = (c_1',\dots,c_k')$ constitutes a \emph{strategy of cops}. Formally, it is a function $q: (r',c,\tau) \mapsto c'$. Performing the moves determined by both strategies gives the new game $\Gamma(r',c',\delta\tau)$ with the fixed position and agility, where $\delta\tau(n) := \tau(n+1)$ ($n\ge1$).

Given the agility $\tau$ and strategies $s,q$ of the robber and the cops, we denote by $v_\tau(s,q)$ the value of the game when it is played using these strategies. Now we define the \emph{guaranteed outcome} for each of the players. First for the robber:
$$
    \ValR(\tau) = \inf_q \sup_s v_\tau(s,q) \quad \textrm{and} \quad  \ValR = \sup_\tau \ValR(\tau).
$$
Similarly for the cops,
$$
    \ValC(\tau) = \sup_s \inf_q v_\tau(s,q) \quad \textrm{and} \quad  \ValC = \sup_\tau \ValC(\tau).
$$
For each $\varepsilon>0$, there is $q$ such that for every $s$, $v_\tau(s,q) < \ValR(\tau)+\varepsilon$. This implies that $$\ValC(\tau) \le \ValR(\tau) \quad \textrm{and} \quad  \ValC \le \ValR.$$
If $\ValC = 0$, then we say that \emph{cops win} the game. If $\ValR>0$, then the \emph{robber wins}.

It is an interesting question whether it can happen that $\ValC < \ValR$ for some game space $X$ and some $k$. In particular, is it possible that both players, the cops and the robber win the game? This question was offered as the main open problem in the afore-mentioned work by Bollob\'as et al.~\cite{BoLeWa12}. For our version of the game, we will answer this question in the negative. Indeed, the subtleties of our definition will allow us to make the conclusion that $\ValC = \ValR$, see Theorem \ref{thm:ValC=ValR}.

Since $X$ is compact, for every $\varepsilon>0$ there exists an integer $k$ such that $k$ cops can always achieve the value of the game be less than $\varepsilon$. (Place the cops at the centers of open balls of radius $\varepsilon$ that cover $X$. Then, no matter where the robber is, he will be at distance less than $\varepsilon$ from one of the cops.) Hence, with the growing number of cops, the value of the game tends to 0.

Given a game space $X$, let $k$ be the minimum integer such that $k$ cops win the game on $X$. This minimum value will be denoted by $c(X)$ and called the \emph{cop number} of $X$. If such a $k$ does not exist, then we set $c(X)=\infty$. Similarly we define the \emph{strong cop number} $c_1(X)$ as the minimum $k$ such that $k$ cops can always catch the robber.

\subsubsection*{Examples:}

{\bf 1. The $n$-ball.}
Let us consider the game with one cop on the $n$-dimensional ball of radius~$1$, $X = B^n = \{x\in \RR^n \mid \Vert x \Vert_2 \le 1\}$. This is the higher-dimensional analogue of the game of Man and Lion. It turns out that the ball of any dimension has cop number 1, i.e., one cop can win the game (although the robber can make sure he is never caught). In this game, the cop can use a strategy from the Man and Lion game (the 2-dimensional version) as follows. First, the cop moves to the center of the ball. From now on he will make sure to always be at the line segment from the center to the current position of the robber. By considering the 2-dimensional plane $\Pi$ through the origin, containing the former and the new position of the robber, he can keep this strategy and approach the robber using the 2-dimensional strategy in $\Pi$ (in which the cop keeps the requirement to be on the line from the center to the robber), thus approaching the robber and achieving infimum of his distance to the robber to be 0.

Note that a slight modification of the described strategy of the cop can be defined so that it only depends on the positions of the players and the current step length, i.e., it does not need the whole information on agility. It goes as follows: If the step duration is $t$ and the cop can reach a point on the segment from the center to the robber, he moves to the point on this segment that is closest to the robber. If he cannot reach the segment in the current step, then he moves distance $t$ towards the center.

The above strategy shows that $c(B^n) = 1$. It can be shown that $c_1(B^n) = n$, see \cite{Croft64,IrMo22}.

\medskip

\noindent
{\bf 2. The $n$-sphere.}
The game with one cop played on the $n$-dimensional sphere of radius $1$ in $\RR^{n+1}$, $S^n = \{x\in \RR^{n+1} \mid \Vert x \Vert_2 = 1\}$ is somewhat different. Here the robber may invoke the following strategy. Let $\tau_\varepsilon$ be the agility in which each step has length $\varepsilon$. The robber can choose this agility and select the initial position of the cop to be at the north pole, while he positions himself at the south pole. At the first step, the robber stays put. Then, the cop moves, and in the next step, the robber can move to the point that is antipodal to the position of the cop. It is easy to see that the minimum distance between the cop and the robber is never below $\pi-\varepsilon$, so the value of the standard game is $\pi$ in this case.

Differential pursuit-evasion game on $S^n$ was studied by Satimov and Kuchkarov \cite{SaKu00}, who proved that $n+1$ cops can catch the robber on $S^n$ and that $n$ cops cannot catch him. Our version of the game has the same outcome, i.e. $c_1(S^n) = n+1$, see \cite{IrMo22}. However, an interesting fact from \cite{IrMo22} is that two cops can win the game, i.e. $c(S^n) = 2$.

\medskip

\noindent
{\bf 3. The cylinder.}
Let $B$ be a game space (a compact geodesic space with intrinsic metric $d$). The \emph{cylinder over $B$} is the geodesic space $X=B\times I$, endowed with the product topology, where $I=[0,1]$ is a real interval of length 1. We consider the $\ell_p$-metric $d_p$ for some $p\ge1$:
$$
    d_p((a,s),(b,t)) = \left(d(a,b)^p + |s-t|^p\right)^{1/p}.
$$
Clearly, $X$ is a geodesic space. For any number $k$ of cops, the value of the game of Cops and Robber on $X$ is the same as on $B$ if $p>1$: $\ValR(X) = \ValR(B)$.
To see this, let us consider any agility $\tau$ and for any $\varepsilon>0$, consider a strategy $s_0$ of the robber in $B$ such that
$$
    \ValC(B,\tau) = \sup_s \inf_q v_\tau(s,q) \le \inf_q v_\tau(s_0,q) + \varepsilon.
$$
The robber can use the same agility and the same strategy on $X$, if he always stays in $B \approx B\times \{0\}$ and considers the cops' positions as being in $B$ by projecting them to the first coordinate. By using the strategy $s_0$, he will be able to keep distance $v_\tau(s_0,q)$ from the cops, where $q$ is the projection of cops' strategy to $B$. Since $\varepsilon$ is arbitrarily small, this shows that $\ValC(B,\tau)\le \ValC(X,\tau)$, and since $\tau$ is arbitrary, it also implies that $\ValC(B)\le \ValC(X)$.

Similarly we see that $\ValR(B)\ge \ValR(X)$. Here, the cops will use a strategy in $B$ to get close to the robber in $B \approx B\times \{0\}$. Once one of the cops is at distance less than $\varepsilon$ from the projection of the robber's position onto $B$, then that cop follows the moves of the robber in the $B$-coordinate and slowly increases its $I$-coordinate with a constant slope, so that after making distance $t$, his $I$-coordinate would have changed by 1, and the length $t$ of his path will be at most $\varepsilon$ larger than the length $t_0$ of the projected path. For instance, we can take $t$ large enough so that $t-(t^p-1)^{1/p} < \varepsilon$. Such $t$ exists since $p>1$. It is easy to see that in this way, the cop will come to a distance less than $2\varepsilon$ from the robber. Again, since $\varepsilon$ is arbitrarily small and $\tau$ is arbitrary, we conclude that $\ValR(B)\ge \ValR(X)$. Both inequalities show that
$\ValC(B) \le \ValC(X) \le \ValR(X) \le \ValR(B)$.
Finally, our Theorem \ref{thm:ValC=ValR} implies that
$$
   \ValC(B) = \ValC(X) = \ValR(X) = \ValR(B).
$$

The same result holds if instead of $X=B\times I$ we consider the product $B\times T$, where $T$ is any metric graph homeomorphic to a tree.

\section{Approximation with a finite game}

The analysis of the game of cops and robber as defined in this article is somewhat complicated due to the fact that the game may have infinitely many steps and that the set of agility functions is not compact. However, the game can be approximated by a finite game within an arbitrary precision. This approximation result is described next.

Suppose that we fix an initial position $(r,c)\in X^{k+1}$, a positive integer $N$ and an agility $\tau$. Then we can consider $N$ steps of the game, and let $T=T_N(\tau)=\sum_{i=1}^N \tau(i)$ be the duration of the game during these $N$ steps. We will use $\tau^N$ to denote the restriction of $\tau$ to the first $N$ values since the rest of $\tau$ is not important for the finite game. This means that we either view $\tau^N$ as a function $[N]\to\RR_+$ or as a function that has $\tau^N(n)=0$ for every $n>N$. We also define $d(r,c) = \min\{d(r,c_i)\mid i\in [k]\}$, $d(c,c') = \max\{d(c_i,c_i')\mid i\in [k]\}$, and $d((r,c),(r',c')) = \max\{d(r,r'),d(c,c')\}$. Given $\tau$, we define its \emph{shift} $\tau_1=\delta\tau$ by the rule $\tau_1(i)=\tau(i+1)$ for $i\ge1$.

For each game $\Gamma(*)$ (where $(*)$ is the defining set of parameters), we will define the value of the game $\Val(\Gamma(*))$, which we will abbreviate by writing simply $\Val(*)$.

Given the finite game $\Gamma(r,c,N,\tau^N)$ with initial position $(r,c)$, its \emph{value} can be defined recursively as follows:
\begin{equation}\label{eq:def_value_finite_game}
  \Val(r,c,N,\tau^N) =
    \left\{
      \begin{array}{ll}
        d(r,c), & \hbox{if $N=0$;} \\[1.5mm]
        \max\limits_{\substack{r'\\d(r,r')\le\tau(1)}}~~\min\limits_{\substack{c'\\d(c,c')\le\tau(1)}} \Val(r',c',N-1,\tau_1^{N-1}), & \hbox{otherwise.}
      \end{array}
    \right.
\end{equation}
The reader will realize that we used maximum and minimum (instead of supremum and infimum) in the definition; this is justified since $X$ is compact and the value of the game is continuous. This fact is formally stated and proved as Lemma \ref{lem:value_is_continuous} below.
Definition (\ref{eq:def_value_finite_game}) then shows that both players have optimal strategies that assure them achieving the value $\Val(r,c,N,\tau^N)$ if they stick with these strategies. Let $s_0: (r,c,N,\tau^N) \mapsto r'$, where $r'$ is the position of the robber for which the maximum in (\ref{eq:def_value_finite_game}) is attained; and for each $r'$, let $q_0: (r',c,N,\tau^N) \mapsto c'$, where $c'$ is a cops' position at which the minimum in (\ref{eq:def_value_finite_game}) is attained. Note that for any strategies $s$ and $q$ of the robber and the cops (respectively), we have:
$$
    v_\tau(s,q_0) \le v_\tau(s_0,q_0) = \Val(r,c,N,\tau^N) \le v_\tau(s_0,q).
$$

There is another small detail about the definition of the value $\Val(r,c,N,\tau^N)$ in (\ref{eq:def_value_finite_game}). Namely, (\ref{eq:def_value_finite_game}) defines the value as the distance $d(r^N,c^N)$ at the end of the game, while the definition of $v_\tau(s,q)$ also considers intermediate distances. This means that the definition should have been introduced as
\begin{equation}\label{eq:def_value_finite_game_intermediate}
  \Val\nolimits'(r,c,N,\tau^N) =
    \left\{
      \begin{array}{ll}
        d(r,c), & \hbox{if $N=0$;} \\[1.5mm]
        \min \bigl\{ d(r,c),\ \max\limits_{r'}~\min\limits_{c'}~ \Val(r',c',N-1,\tau_1^{N-1}) \bigr\}, & \hbox{otherwise}
      \end{array}
    \right.
\end{equation}
where the maximum and the minimum are taken over the same nearby points $r'$ and $c'$ as in (\ref{eq:def_value_finite_game}).
However, this definition is equivalent as shown below.

\begin{lemma}\label{lem:L1}
  For any finite-step game $\Gamma(r,c,N,\tau^N)$, we have $\Val(r,c,N,\tau^N) = \Val'(r,c,N,\tau^N)$.
\end{lemma}

\begin{proof}
The proof is by induction on the number of steps. By using the induction hypothesis, the only way that $\Val' \ne \Val$ is that $d(r,c)<\Val(r',c',N-1,\tau_1^{N-1})$. However, if that were the case, the cop at distance $d(r,c)$ from the robber could just follow the robber throughout the game, and achieve the value $d(r,c)$, which is a contradiction to the assumption that $d(r,c) < \Val(r',c',N-1,\tau_1^{N-1})$.
\end{proof}

A corollary of Lemma \ref{lem:L1} is that the value of the game is non-increasing in terms of its number of steps, which we state formally below.

\begin{lemma}\label{lem:value decreases with more steps}
  For any agility $\tau$, initial position $(r,c)$ and positive integers $N\le M$, we have
  $$\Val(r,c,N,\tau^N) \ge \Val(r,c,M,\tau^M).$$
\end{lemma}

For our later use we now define the value of a (non-finite) game with given agility:
\begin{equation}\label{eq:Value(r,c,tau) definition}
  \Val(r,c,\tau) = \lim_{N\to\infty} \Val(r,c,N,\tau^N).
\end{equation}
The limit exists since the value of the game decreases with $N$.

Our goal is to approximate the game with finite games, and we use the following as the definition what it means to approximate the game.
The game $\Gamma(r,c,\tau)$ is \emph{$\varepsilon$-approximated} with a finite game $\Gamma(r,c,N,\tau^N)$ if $\Val(r,c,N,\tau^N) - \Val(r,c,\tau) < \varepsilon$. Note that by Lemma \ref{lem:value decreases with more steps}, the difference $\Val(r,c,N,\tau^N) - \Val(r,c,\tau)$ cannot be negative.

We are ready to proceed with a proof that the game value is continuous.

\begin{lemma}\label{lem:value_is_continuous}
The value $\Val(r,c,N,\tau^N)$ of a finite game is continuous in terms of its initial position $(r,c)$ and $\tau^N$. More precisely, if $d((r,c),(r',c')) \le \delta$, then
$$\left|\Val(r,c,N,\tau^N) - \Val(r',c',N,\tau^N)\right| \le 2\delta$$
and if\/ $\Vert\tau^N-{\tau'}^N\Vert_1 := \sum_{n=1}^N |\tau(n)-\tau'(n)| \le \varepsilon$, then
$$\left|\Val(r,c,N,\tau^N) - \Val(r,c,N,{\tau'}^N)\right| \le 2\varepsilon.$$
\end{lemma}

\begin{proof}
Suppose that $d((r,c),(r',c')) \le \delta$. Suppose that the starting position is $(r',c')$. We say that the robber \emph{mimics} the game strategy for the initial position $(r,c)$ if he plays so that he is always within distance $\delta$ from the position he would have in the game when starting with $(r,c)$. The first move is now obvious, the robber would move to a point $x$ at distance at most $\tau(1)$ from $r$. Since $d(r',r)\le\delta$, he can move to a point $r''$ that is within distance $\delta$ from $x$. Now the cops move to their new positions $c_i''$. Since $c_i'$ was within distance $\delta$ from $c_i$, there is a point $y_i$ at distance at most $\delta$ from $c_i''$ and at distance at most $\tau(1)$ from $c_i'$. The robber considers the position $(x,y_1,\dots,y_k)$ as the imaginary position in this step and uses the strategy for the game $\Gamma(r,c,N,\tau^N)$ game as being at this position in the current step. This shows that the distance between $x$ and $y$ is all the time at least $\Val(r,c,N,\tau^N)$, and thus
$$d(r'',c'') \ge \Val(r,c,N,\tau^N) - 2\delta.$$
This implies that
$\Val(r',c',N,\tau^N) \ge \Val(r,c,N,\tau^N) - 2\delta$. To obtain the inequality in the other way, just switch the roles of $(r,c)$ and $(r',c')$.

Similarly we prove continuity with respect to $\tau$. Here we assume that $\sum_{n=1}^N |\tau(n)-\tau'(n)| \le \varepsilon$. In the game $\Gamma(r,c,N,{\tau'}^N)$, either player can mimic his strategy for $\Gamma(r,c,N,\tau^N)$, and all the time being at most $\varepsilon$ away from the imaginary positions in the gameplay of $\Gamma(r,c,N,\tau^N)$. This implies the stated inequality.
\end{proof}

Having defined the value of a finite-step game by (\ref{eq:def_value_finite_game}) (or, equivalently, by (\ref{eq:def_value_finite_game_intermediate})), let us fix $T>0$ and consider all finite-time games $\Gamma(r,c,T)$ with duration $T$. Then we define
\begin{equation}\label{eq:game_value_fixed_T}
  \Val(r,c,T) = \sup_{\substack{N,\tau\\T_N(\tau) = T}} \Val(r,c,N,\tau^N)
\end{equation}
and for the standard game $\Gamma(r,c)$, we define:
\begin{equation}\label{eq:game_value_T_grows}
  \Val(r,c) = \inf_{T \to \infty} \Val(r,c,T).
\end{equation}
Let us observe that by Lemma \ref{lem:value decreases with more steps}, $\Val(r,c,T)$ is non-increasing with $T$, thus the infimum in (\ref{eq:game_value_T_grows}) can also be replaced by a limit when $T\to \infty$.
Instead of (\ref{eq:game_value_T_grows}), we could as well have used (\ref{eq:Value(r,c,tau) definition}) since we have
$$
   \Val(r,c) = \sup_\tau \Val(r,c,\tau).
$$

\section{Choice of agility functions}

It is helpful to know that the ``approximately-best'' agility (which is initially chosen by the robber) may be assumed to be decreasing. This is an easy consequence of the fact that by subdividing a step of duration $t$ into two or more steps of the same total duration will not decrease the value of the game. Formally this is settled by the following result.

Let $\tau$ be an agility, let $0\le\alpha\le1$ and $i\ge 1$. Let $\sigma_i^\alpha \tau$ be the agility defined by the rule:
\begin{equation}
  \sigma_i^\alpha \tau(j) =
     \left\{
       \begin{array}{ll}
         \tau(j), & \hbox{if $1\le j<i$;} \\
         \alpha\tau(i), & \hbox{if $j=i$;} \\
         (1-\alpha)\tau(i), & \hbox{if $j=i+1$;} \\
         \tau(j-1), & \hbox{if $j\ge i+2$.}
       \end{array}
     \right.\label{eq:subdivide agility}
\end{equation}
We say that the agility $\sigma_i^\alpha\tau$ has been obtained from $\tau$ by an \emph{elementary subdivision} of the $i$th step. Agility $\tau'$ is a \emph{subdivision} of $\tau$ if it can be obtained from $\tau$ by a series of elementary subdivisions. Here we allow infinitely many elementary subdivisions, but it is requested that each step of $\tau$ is subdivided into finitely many substeps in order to obtain $\tau'$. We write $\tau' \preceq \tau$ if $\tau'$ is a subdivision of $\tau$.

For further reference we state the following easy observation, whose proof is left to the reader.

\begin{lemma}\label{lem:common subdivision}
  Any two agility functions have a common subdivision.
\end{lemma}

The following lemma shows that using $\tau'$ instead of $\tau$ goes to the favour of the robber. As a consequence, we may assume that the robber always chooses a decreasing agility.

\begin{lemma}\label{lem:subdivide tau}
Suppose that $\tau'$ is a subdivision of an agility $\tau$. Then
$$\Val(r,c,\tau) \le \Val(r,c,\tau').$$
\end{lemma}

\begin{proof}
It suffices to prove the result for finite-step games, and for each such game $\Gamma(r,c,N,\tau^N)$, it suffices to prove that an elementary subdivision does not increase the value of the game, i.e.
$$\Val(r,c,N,\tau^N) \le \Val(r,c,N+1,\sigma_i^\alpha\tau^{N+1}).$$
To see this, consider steps $i$ and $i+1$ of the game with the subdivided agility. The robber can just follow the optimal strategy of the $i$th step from the original game $\Gamma(r,c,N,\tau^N)$ and move to the desired position in two steps. It is easy to see that this cannot give him smaller value of the game.
\end{proof}

\begin{lemma}\label{lem:approx_RC}
For every initial position $(r,c)$ and every $\varepsilon>0$, there is an integer $N$ and an agility $\tau^N$ such that for any initial position $(r,c)$ we have:
$$
   \Val(r,c)-\varepsilon < \Val(r,c,N,\tau^N) < \Val(r,c)+\varepsilon.
$$
\end{lemma}

\begin{proof}
Take $T$ large enough so that $\Val(r,c,T) \le \Val(r,c)+\varepsilon$ and then choose $N,\tau^N$ so that $T_N(\tau)=T$ and
$\Val(r,c,N,\tau^N) \ge \Val(r,c,T)-\varepsilon$. This implies that $\Val(r,c,N,\tau^N)\ge \Val(r,c,T)-\varepsilon \ge \Val(r,c)-\varepsilon$ and $\Val(r,c,N,\tau^N)\le \Val(r,c,T) \le \Val(r,c)+\varepsilon$.
\end{proof}

\begin{corollary}\label{cor:approx_allRC}
For every $\varepsilon>0$ and $T>0$, there exists a finite game with $N$ steps and finite agility $\tau^N$ with $T_N(\tau)\ge T$ such that for every initial position $(r,c)$, we have
$$
   \Val(r,c)-\varepsilon < \Val(r,c,N,\tau^N) < \Val(r,c)+\varepsilon.
$$
\end{corollary}

\begin{proof}
The goal is to show that there are $N$ and $\tau^N$ that approximate the game for any initial position. We will say that $N,\tau^N$ give \emph{$\varepsilon$-approximation} for $(r,c)$ if the inequalities of the corollary hold for the initial position $(r,c)$. Since $X$ and hence also $X^{k+1}$ is compact, there is a finite set $Y$ of initial positions such that any other initial position is within distance $\varepsilon/4$ from one of the positions in $Y$. For each $(r,c)\in Y$, there are $N$ and $\tau^N$ that give $\varepsilon/2$-approximation. By Lemma \ref{lem:common subdivision} there is a common subdivision of all such agility functions $\tau^N$. Since by subdividing the agility, the value only increases (Lemma \ref{lem:subdivide tau}), we may assume that the same pair $N,\tau^N$ $\varepsilon/2$-approximates every initial position $(r,c)\in Y$. Since any other position is at distance at most $\varepsilon/4$ from $Y$, Lemma \ref{lem:value_is_continuous} implies that the finite game with $N,\tau^N$ $\varepsilon$-approximates every initial position.
\end{proof}

Corollary \ref{cor:approx_allRC} combined with Lemma \ref{lem:value_is_continuous} implies that the value $\Val(r,c)$ of any Cops and Robber game is a continuous function, depending on the initial position. This in particular shows that the set of cop-winning initial positions is closed.

The use of agility functions is the main technical reason that makes the game ``non-compact". Usually, there will be no optimal agility since the set of all agility functions is not closed. However, we can make some general observations that will enable us to use certain assumptions about ``near-optimal" agility.

We say that $\tau$ is \emph{decreasing} if $\tau(n+1)<\tau(n)$ for every $n\ge1$. It is easy to see by using induction that for every $\tau$ there is a decreasing agility $\tau'$ such that $\tau'\preceq\tau$. The following lemma shows that using $\tau'$ instead of $\tau$ goes to the favour of the robber, and therefore, we may assume that the robber always chooses a decreasing agility.

For each agility $\tau\in\Sigma^0$, there is a decreasing agility $\tau'\in \Sigma^0$ that is a subdivision of $\tau$. The following statement enables us to restrict our attention to decreasing agility functions

\begin{lemma}\label{lem:subdivide strategies}
Suppose that $\tau\in\Sigma^0$ is an agility function and that $\tau'$ is a subdivision of $\tau$ that is decreasing.

(a) If $s: X^{k+1} \to X$ is a strategy of the robber with agility function $\tau$,
then there is a strategy $s': X^{k+1} \to X$ for agility $\tau'$ such that
\begin{equation}\label{eq:subdividing agility better for R}
   \inf_{q'} v_{\tau'}(s',q') \ge \inf_{q} v_{\tau}(s,q),
\end{equation}
where $q$ and $q'$ in both infima run over all strategies of the cops with agility $\tau$ and $\tau'$, respectively.

(b) For any position $(r,c)$, $\Val(r,c,\tau') \ge \Val(r,c,\tau)$.
\end{lemma}

\begin{proof}
(a) For the subdivided $n$th step of $\tau$, the agility $\tau'$ provides a finite number of steps, whose total duration is equal to $\tau(n)$. If the strategy $s$ tells the robber to move from position $r$ to $r'$, the robber can do the same with subdivided steps, and his strategy $s'$ will move from $r$ through all substeps to $r'$. Formally, there could have been a problem with this if at some later substep would need to make the same decision. But since $\tau'$ is decreasing, the strategy for the same position will have smaller values in the agility function, so any later substeps can use different strategy values.

To prove (\ref{eq:subdividing agility better for R}), we just argue that if the robber, when using strategy $s$, is able to keep distance $v$ from the cops for any cops' strategy $q$, he is also able to keep the same distance by using his strategy $s'$. Suppose not. Then the cops have a strategy $q'$ achieving better outcome. They can mimic that strategy under $\tau$. If they gain better distance in a substep, they can keep that distance by following the robber.

(b) This is an immediate corollary of the definition (\ref{eq:Value(r,c,tau) definition}) of $\Val(r,c,\tau)$ and of part (a) of this lemma.
\end{proof}

\section{Volatile games}

Suppose that we have nonnegative real numbers $(\varepsilon_n)_{n\ge0}$ and that two players play the game of cops and robber in $X$. Suppose that after completing each step $n$ of the game ($n\ge0$), an \emph{adversary} changes the positions of the robber and of the cops by moving them to points that are at most $\varepsilon_n$ away from their current position. Then we say that the players play a \emph{volatile game} with \emph{perturbation $(\varepsilon_n)_{n\ge0}$}.

Lemma \ref{lem:value_is_continuous} can be interpreted as making a small change in the position before making step 1 (which is ``after making step 0''), thus it corresponds to a volatile game where $\varepsilon_0=\varepsilon$ and $\varepsilon_n=0$ for $n\ge 1$. For a general finite-step volatile game $\Gamma(r,c,N,\tau^N,(\varepsilon_n)_{n\ge0})$ we define the guaranteed value for the robber as the minimum distance from the cops over all possible strategies, when the adversary ``helps the cops'' as much as he can. This can be defined recursively in the same way as in (\ref{eq:def_value_finite_game}):

\begin{equation}\label{eq:def_value_volatile_game}
  \ValC(r,c,N,\tau^N,(\varepsilon_n)_{n\ge0}) =
    \left\{
      \begin{array}{ll}
        \max\{d(r,c)-2\varepsilon_0,0\}, & \hbox{if $N=0$;} \\[1.5mm]
        \min\limits_{r',c'}\,
        \max\limits_{r''}\,
        \min\limits_{c''}\, \ValC(r'',c'',N-1,\tau_1^{N-1},(\varepsilon'_n)_{n\ge0})), & \hbox{otherwise;}
      \end{array}
    \right.
\end{equation}
where $\varepsilon'_n = \varepsilon_{n+1}$ for $n\ge0$ and the minima and the maximum are taken over all $r',c',r'',c''$ for which
$d(r,r')\le \varepsilon_0$, $d(c_i,c'_i)\le \varepsilon_0$, $d(r',r'')\le\tau(1)$, and $d(c'_i,c''_i)\le\tau(1)$ ($i\in [k]$).
In the definition (the case $N=0$) we have used the fact that at the last step of the game, the adversary can move the robber and its closest cop towards each other by the distance of maximum allowed perturbation in this step. For the recursive step we also used the possibility of the worst perturbation of the positions made by the adversary.

Similarly, we define the guaranteed value for the cops as the minimum distance from the robber over all possible strategies. This can be defined recursively in the same way as in (\ref{eq:def_value_volatile_game}):

\begin{equation}\label{eq:def_value_volatile_game_cops}
  \ValR(r,c,N,\tau^N,(\varepsilon_n)_{n\ge0}) =
    \left\{
      \begin{array}{ll}
        d(r,c), & \hbox{if $N=0$;} \\[1.5mm]
        \max\limits_{r',c'}\, \max\limits_{r''}\, \min\limits_{c''}\, \ValR(r'',c'',N-1,\tau_1^{N-1},(\varepsilon'_n)_{n\ge0})), & \hbox{otherwise.}
      \end{array}
    \right.
\end{equation}
The difference here is that we have estimated the worst way for the cops when the adversary will try to ``help the robber''.

Clearly, $\ValC$ and $\ValR$ for volatile games need not be the same, and it may also happen that $\ValC>\ValR$. However, the two values cannot be too far from each other as we prove next.

\begin{lemma}\label{lem:value volatile game}
  Let $\Gamma(r,c,N,\tau^N)$ be a finite-step game and let $\Gamma(r,c,N,\tau^N,(\varepsilon_n)_{n\ge0})$ be a volatile version. Set $\delta_n = \sum_{i=0}^n \varepsilon_i$. Then
  $$
     \ValC(r,c,N,\tau^N,(\varepsilon_n)_{n\ge0}) \ge \Val(r,c,N,\tau^N) - 2\delta_N
  $$
  and
  $$
     \ValR\left(r,c,N,\tau^N,(\varepsilon_n)_{n\ge0}\right) \le \Val(r,c,N,\tau^N) + 2\delta_{N-1}.
  $$
\end{lemma}

\begin{proof}
  The proof is not complicated, but requires thoughtful description. We will only show the first inequality. Here it suffices to describe a strategy of the robber such that no matter how the cops play, the final result will in the worst case be close to the value $\Val(r,c,N,\tau^N)$ of the finite game.

  Let $s_0$ be the optimal strategy for the game $\Gamma(r,c,N,\tau^N)$. It suffices to prove by induction on $N$ that the robber can stay ``close'' to the gameplay in the unperturbed game in which he uses strategy $s_0$ and the cops use some strategy $q$ since in that gameplay the distance between the cops and the robber will always be at least $\Val(r,c,N,\tau^N)$. By ``close'' we mean that his position and the position of each cop after $n$ steps is at most $\delta_n$ away from a position in the usual gameplay with the robber using strategy $s_0$.

  When $N=0$, this is clear. If $N>0$, use the induction hypothesis to conclude that after $N-1$ steps, we are $\delta_{N-1}$-close to a position $(r^{N-1},c^{N-1})$ of the unperturbed gameplay. By mimicing the last step of the game, the players can come to within the distance $\delta_{N-1}$ from the position $(r^N,c^N)$ in the optimal gameplay. Finally, the adversary can move each player up to distance $\varepsilon_N$ further away from that position. Since $\delta_{N-1}+\varepsilon_N=\delta_N$, players could be at most $2\delta_N$ closer than in the optimal gameplay performed by the robber.
  This completes the proof.
\end{proof}

\section{Limits of strategies}

The somewhat technical definition of our Cops and Robber games using agility functions enables us to prove that each game $\Gamma(r,c)$ or $\Gamma(r,c,\tau)$ can be approximated by finite discrete games of the form $\Gamma(r,c,T_i,\tau)$, where $T_i$ is an increasing sequence of positive real numbers tending to infinity. When $\tau$ is fixed, we can equivalently approximate $\Gamma(r,c,\tau)$ with finite-step games $\Gamma(r,c,N_i,\tau^{N_i})$, where $0<N_1<N_2<N_3<\cdots$ is a sequence of positive integers tending to infinity.
However, strategies are not continuous functions, so there is an obvious question if one can define a limiting strategy, knowing optimal strategies for finite games that are $\varepsilon$-approximations with $\varepsilon\to 0$. This question is our goal in this section. Answering what are the limits of strategies will also yield our final min-max theorem.

Suppose that $\tau$ is an agility that is decreasing. Let $0<T_1<T_2<\cdots$ be an increasing sequence tending to infinity. For each $i$, let $N_i$ be the smallest integer such that $T(N_i) := \sum_{n=1}^{N_i} \tau(n) \ge T_i$.
If we start instead with an increasing sequence of positive integers tending to infinity, $N_1<N_2<N_3<\cdots$, we can define $T_i=T(N_i)$ and then proceed as we would with the sequence $(T_i)_{i\ge1}$.
We will also use, with a fixed value of $\varepsilon$, the following decreasing sequence: $\varepsilon_i := 2^{-i-1}\varepsilon$ ($i\ge0$). Note that $\sum_{i\ge0} \varepsilon_i = \varepsilon$.

Let us recall that by Corollary \ref{cor:approx_allRC}, for every $\varepsilon>0$ and every agility $\tau$, there is $T=T(\varepsilon)$ such that $\Gamma(r,c,\tau)$ is $\varepsilon$-approximated with the finite game $\Gamma(r,c,T,\tau)$ for every initial position $(r,c)$. Let us now choose $T_i$ ($i\ge 1$) so that $\Gamma(r,c,\tau)$ is $\varepsilon_i$-approximated with the finite game $\Gamma(r,c,T_i,\tau) = \Gamma(r,c,N_i,\tau^{N_i})$. For each $i\ge1$, let $s_i$ and $q_i$ be optimal strategies of the robber and the cops (respectively) for the finite game $\Gamma(r,c,N_i,\tau^{N_i})$.

We now define the notion of an $\varepsilon$-approximate limit of strategies $(s_i)_{i\ge1}$ and $(q_i)_{i\ge1}$ for the game $\Gamma(r,c,\tau)$. Since $\tau$ is decreasing, we can define the limiting strategies depending only on the current step length $\tau(n)$, and this is defined first for $n=1$, then for $n=2$, etc. Suppose that strategies have been defined for steps $1,2,\dots,n-1$. Consider the next step, whose length is $t=\tau(n)$. Consider a finite point-set $Y^n\subset X^{k+1}$ such that every point in $X^{k+1}$ is at distance less than $\varepsilon_n$ from $Y^n$. For every position $(r^n,c^n)\in Y^n$, each strategy $s_i$ ($i\ge1$) gives the new position $r_i^n$ of the robber that is at most $t$ away from $r^n$. Since $X$ is compact, the sequence $(r_i^n)_{i\ge1}$ contains a convergent subsequence with limit $r'$, and $r'$ is still at most $t$ away from $r^n$. We say that this subsequence is \emph{appropriate} for $(r^n,c^n)$ if, moreover, $d(r_i^n,r') < \varepsilon_n$ for each $i\ge1$. Since $Y^n$ is finite, there is a subsequence $j_1^n < j_2^n < j_3^n < \dots$ of indices that is appropriate for all points in $Y^n$. We will from now on, in all succeeding steps, only consider this subsequence and its subsequences; so in particular, if $n>1$, then $(j_i^n)_{i\ge1}$ is a subsequence of $(j_i^{n-1})_{i\ge1}$. The corresponding lengths of the finite games will also be denoted using superscripts $n$, $N_i^n := N_{j_i^n}$ ($i\ge1$). The assignment $s_0: (r^n,c^n,t)\mapsto r'$ then defines the strategy of the robber in step $n$ restricted to the positions in $Y^n$.

If a position $(r^n,c^n)$ is not in $Y^n$, then we find the closest position, say $(a,b)\in Y^n$, and consider its strategy limit point $r'=s_0(a,b,t)$. Now, we first move the robber from $r^n$ to $a$ and then proceed towards $r'$ using a geodesic from $a$ to $r'$. Of course, the total length we can make is $t$, and we may need to stop before reaching $r'$. Let $r''$ be the point reached in this way, and we set $s_0: (r^n,c^n,t)\mapsto r''$.
Observe that we also have $d(r'',r')\le \varepsilon_n$. This defines the strategy of the robber for the $n$th step.

Similarly we define a strategy $q_0:(r,c,t)\mapsto c')$ for the cops in the $n$th step of length $t$, again passing to a subsequence $(N_i^n)$ of $(N_i^{n-1})$. We may assume that the same subsequence that was appropriate for all $(r^n,c^n)\in Y^n$ for the robber is also appropriate with respect to the cops' positions, so we do not need to introduce new notation.

By repeating the above process for each $n\ge 1$, we end up with strategies $s_0$ and $q_0$ for the game $\Gamma(\tau)$, where $\tau$ is a fixed decreasing agility. Note that there is some freedom in the process by taking different converging subsequences, so these need not be unique. Any strategies $s_0$ and $q_0$ obtained in this way are said to be \emph{$\varepsilon$-approximate limits of strategies} $(s_i)_{i\ge1}$ and $(q_i)_{i\ge1}$ for the finite games $(\Gamma(r,c,N_i,\tau^{N_i}))_{i\ge1}$.

\begin{lemma}\label{lem:approximate limits approximate the value}
   Let $\tau$ be a decreasing agility function and $(r^0,c^0)$ be a fixed initial position for the game of cops and robber on $X$. Suppose that $s_0$ and $q_0$ are $\varepsilon$-approximate limits of optimal strategies for finite games $\Gamma(r^0,c^0,N_i,\tau^{N_i})$ with $N_i\to \infty$ as $i\to\infty$. Then for arbitrary strategies, $s$ of the robber and $q$ of the cops, for the game $\Gamma(r^0,c^0,\tau)$, we have
   $$
      v_\tau(s,q_0) - 4\varepsilon < v_\tau(s_0,q_0) < v_\tau(s_0,q) + 4\varepsilon.
   $$
\end{lemma}

\begin{proof}
  We will give details only for the second inequality. The first one can be proved in exactly the same way, using estimates on the positions of each of the cops.

  Suppose that the players start with position $(r^0,c^0)$ and then play the game using agility $\tau$, where the robber uses the limit strategy $s_0$ and the cops use any strategy $q$. When defining $s_0$ as the $\varepsilon$-approximate limit strategy in the $n$th step, we have considered a finite point set $Y^n$ such that each point in $X^{k+1}$ was at distance less than $\varepsilon_n$ from $Y^n$, and we considered a subsequence $(N_i^n)_{i\ge1}$ of the (sub)sequence $(N_i^{n-1})_{i\ge1}$ from the previous step. For $N_i^n$, the optimal strategies of the finite games $\Gamma(r^0,c^0,N_i^n,\tau^{N_i^n})$ converge for each $(r,c,t)$ with $(r,c)\in Y^n$ and $t=\tau(n)$ to the point $r'$ which we have used when defining the strategy $s_0$.

  We can interpret the gameplay as a volatile game, where the adversary first moves the robber from where he would have ended using his optimal strategy to the limiting value $r'$. By our choice of subsequences, this moves the robber at most $\varepsilon_n$ away in the $n$th step. Next, the adversary moves both, the robber and the cops, each at most $\varepsilon_n$ away, by placing them to the nearest position $(a,b)\in Y^n$. Thus, the first $n$ steps of the gameplay of \emph{each} finite game $\Gamma(r^0,c^0,N_i^n,\tau^{N_i^n})$ ($i\ge1$) corresponds to a volatile game with perturbations $(2\varepsilon_n)_{n\ge1}$. By Lemma \ref{lem:value volatile game}, the value $v_\tau(s_0,q)$ of this gameplay is within $2\delta_{N_i^n}$ from $\Val(r^0,c^0,N_i^N,\tau^{N_i^n})$, where $\delta_n = \sum_{j=1}^{n} 2\varepsilon_j < 2\varepsilon$. This completes the proof.
\end{proof}

Our final results are now immediate consequences of Lemma \ref{lem:approximate limits approximate the value} when combined with the definitions of $\ValC(\tau)$ and $\ValR(\tau)$.
The lemma gives the min-max theorem for $\Gamma(r,c,\tau)$, where $(r,c)$ is any initial position. If the approximate strategies for different initial values in the same step $n$ were different, we can unify them so that they will be the same. Having done this, we may assume that we have $\varepsilon$-approximate strategies for all initial positions, forming the $\varepsilon$-approximate strategy for the game $\Gamma(\tau)$. This gives the following result.

\begin{corollary}
  Suppose that $\tau$ is a decreasing agility function and that $s_0$ and $q_0$ are $\varepsilon$-approximate limits of optimal strategies for finite games $\Gamma(N_i,\tau)$ with $N_i\to \infty$ as $i\to\infty$. Then
  $$\ValR(\tau)-\varepsilon \le v_\tau(s_0,q_0)\le \ValC(\tau) + \varepsilon.$$
\end{corollary}

Finally, taking $\varepsilon \to 0$, we get a min-max result.

\begin{theorem}\label{thm:ValC=ValR}
  For every decreasing agility $\tau$ we have $\ValC(\tau)=\ValR(\tau)$. Consequently, $\ValC=\ValR$.
\end{theorem}

The above ``min-max theorem" implies that for every $\varepsilon>0$ there are $\varepsilon$-approximating strategies for both players, and if either one of them uses his strategy, the other player cannot do more than $\varepsilon$ better than just using his own $\varepsilon$-approximating strategy.

\bibliographystyle{plain}
\bibliography{references_CRsurfaces}

\end{document}